\documentclass[12pt, reqno]{amsart}
\usepackage{amsmath, amsthm, amscd, amsfonts, amssymb, graphicx, color}
\usepackage[bookmarksnumbered, colorlinks, plainpages]{hyperref}

\newtheorem{theorem}{Theorem}[section]
\newtheorem{lemma}[theorem]{Lemma}

\newtheorem{corollary}[theorem]{Corollary}
\theoremstyle{definition}

\newtheorem{example}[theorem]{Example}

\theoremstyle{remark}

\numberwithin{equation}{section}

\begin{document}
\setcounter{page}{1}


\title[characterization of Jordan centralizers and Jordan two-sided ...]{characterization of Jordan centralizers and Jordan two-sided centralizers on triangular rings without
assuming unity}

\author[A. Hosseini et al.]{Amin Hosseini$^*$ et al.}

\address{ Amin Hosseini, Department of Mathematics, Kashmar Higher Education Institute- Kashmar- Iran}
\email{\textcolor[rgb]{0.00,0.00,0.84}{hussi.kashm@gmail.com}}





\date{Received: xxxxxx; Revised: yyyyyy; Accepted: zzzzzz.
\newline \indent $^{*}$ Corresponding author}

\begin{abstract}
The main purpose of this paper is to show that every Jordan centralizer and every Jordan two-sided centralizer is a centralizer on triangular rings without assuming unity. As an application, we prove that every Jordan generalized derivation on a triangular ring is a two-sided generalized derivation. Some other related results are also discussed.
\end{abstract} \maketitle

\section{Introduction and preliminaries}
Let $\mathcal{R}$ be an associative ring. An additive mapping $T:\mathcal{R} \rightarrow \mathcal{R}$ is said to be a \emph{left} (resp. \emph{right}) \emph{centralizer} or \emph{multiplier} of $\mathcal{R}$ if $T(xy) = T(x)y$ (resp. $T(xy) = x T(y)$) for all $x, y \in \mathcal{R}$. We call $T$ a \emph{two-sided centralizer} (some authors call it a centralizer) if it is both left and right centralizer. An additive mapping $T$ on $\mathcal{R}$ is called a \emph{Jordan left} (resp. \emph{right}) centralizer if $T(x^2) = T(x) x$ (resp. $T(x^2) = x T(x)$) for all $x \in \mathcal{R}$, and it is called a \emph{Jordan two-sided centralizer} if $T(x^2) = T(x)x = x T(x)$ for all $x \in \mathcal{R}$. An additive mapping $T$ on $\mathcal{R}$ is called a \emph{Jordan centralizer} if $T(x \circ y) = T(x) \circ y = x \circ T(y)$ for all $x, y \in \mathcal{R}$, where $x \circ y = xy + yx$. Since the product $\circ$ is commutative, there is no difference between the left and right Jordan centralizers. Obviously, every left (resp. right) centralizer is a Jordan left (resp. right) centralizer. But the converse is, in general, not true. For instance, let $\mathcal{S}$ be a ring such that the square of each element of $\mathcal{S}$ is zero, but the product of some nonzero elements of $\mathcal{S}$ is nonzero. Let
\begin{align*}
\mathfrak{S} = \Bigg\{\left [\begin{array}{ccc}
0 & a & b\\
0 & 0 & a\\
0 & 0 & 0
\end{array}\right ] \ : \ a, b \in \mathcal{S}\Bigg\}
\end{align*}
Clearly, $\mathfrak{S}$ is a ring. Define the mapping $T:\mathfrak{S} \rightarrow \mathfrak{S}$ by $$T\Bigg(\left [\begin{array}{ccc}
0 & a & b\\
0 & 0 & a\\
0 & 0 & 0
\end{array}\right ]\Bigg) = \left [\begin{array}{ccc}
0 & a & 0\\
0 & 0 & a\\
0 & 0 & 0
\end{array}\right ].$$
A straightforward verification shows that $T$ is both a Jordan left- and a right centralizer on $\mathfrak{S}$, but it is neither a left centralizer nor a right centralizer.

It is well known that centralizers are very important both in theory and applications, and so extensive studies have been conducted on them so far. Centralizers have been studied in the general framework of prime rings and semiprime rings. Recall that a ring $\mathcal{R}$ is called prime if for $x, y \in \mathcal{R}$, $x \mathcal{R} y = \{0\}$ implies that $x = 0$ or $y = 0$, and is semiprime if for $x \in \mathcal{R}$, $x \mathcal{R} x = \{0\}$ implies that $x = 0$. A ring $\mathcal{R}$ is said to be $n$-torsion free, where $n > 1$ is an integer, if for $x \in \mathcal{R}$, $nx = 0$ implies that $x = 0$. Similarly, a module $\mathcal{M}$ is said to be $n$-torsion free, where $n > 1$ is an integer, if for $\mathfrak{m} \in \mathcal{M}$, $n\mathfrak{m} = 0$ implies that $\mathfrak{m} = 0$.

In general, the question under what conditions a map becomes a left (resp. right) centralizer or a two-sided centralizer attracted much attention of mathematicians. Zalar \cite{Z} proved that if $\mathcal{R}$ is a 2-torsion free semiprime ring and $T$ is a Jordan left (resp. right) centralizer on $\mathcal{R}$, then $T$ is a left (resp. right) centralizer. In the same article he proved that every Jordan centralizer of a 2-torsion free semiprime ring $\mathcal{R}$ is a centralizer. Vukman \cite{V7} showed that an additive mapping $T:\mathcal{R} \rightarrow \mathcal{R}$, where $\mathcal{R}$ is a 2-torsion free semiprime ring, satisfying $2 T(x^2) = T(x)x + x T(x)$ for all $x \in \mathcal{R}$, is a centralizer. 
To read more about centralizers, we refer the readers to some recent papers \cite{A, C, F1, G, K1, K2, L, V, V5, V6, Z}, where further references can be found.
So far, many mathematicians have investigated Jordan centralizers on triangular rings (or algebras) with unity, see, e.g. \cite{C1, G1, H, L1, L2, F2}, and the references therein. In the present paper, we shall initiate the study of Jordan centralizers, Jordan two-sided centralizers and some related mappings on triangular rings without assuming unity.  \\

Triangular rings play a fundamental role in this paper. Let $\mathcal{A}$ and $\mathcal{B}$ be two associative rings and let $\mathcal{M}$ be an $(\mathcal{A}, \mathcal{B})$-bimodule which is faithful as a left $\mathcal{A}$-module as well as a right $\mathcal{B}$-module. The ring
\begin{align*}
\mathfrak{T} = Tri(\mathcal{A}, \mathcal{M}, \mathcal{B}) := \left \{\left [\begin{array}{cc}
a & m\\
0 & b
\end{array}\right ] \ : \ a \in \mathcal{A}, \ b \in \mathcal{B}, \ m \in \mathcal{M} \right\}
\end{align*}
under the usual matrix operations is said to be a \emph{triangular ring}. The important examples of triangular rings are upper triangular matrices over a ring $\mathcal{A}$, block upper triangular matrix algebras and nest algebras. It is clear that $\mathfrak{T} = Tri(\mathcal{A}, \mathcal{M}, \mathcal{B})$ is unital if and only if both $\mathcal{A}$ and $\mathcal{B}$ are unital. Note that there exist many triangular rings without unity. For instance, if $\mathcal{A}$ is a ring without unity, then the set of all upper triangular matrices over $\mathcal{A}$ is a ring without unity.

We set
\begin{align*}
& \mathfrak{T}_{11} =  \left \{\left [\begin{array}{cc}
a & 0\\
0 & 0
\end{array}\right ] \ : \ a \in \mathcal{A} \right\}, \\ & \mathfrak{T}_{12} =  \left \{\left [\begin{array}{cc}
0 & m\\
0 & 0
\end{array}\right ] \ : \ m \in \mathcal{M} \right\}, \\ & \mathfrak{T}_{22} =  \left \{\left [\begin{array}{cc}
0 & 0\\
0 & b
\end{array}\right ] \ : \ b \in \mathcal{B} \right\}.
\end{align*}
Then we may write $\mathfrak{T} = \mathfrak{T}_{11} \bigoplus \mathfrak{T}_{12} \bigoplus \mathfrak{T}_{22}$ and every element $X \in \mathfrak{T}$ can be written as $X = X_{11} + X_{12} + X_{22}$, where $X_{ij} \in \mathfrak{T}_{12}$, $i, j \in \{1, 2\}$.

Following \cite{F} and \cite[Lemma 1.5]{J}, we say that a ring $\mathcal{R}$ satisfies \emph{Condition (P)} whenever $x a x = 0$, for all $x \in \mathcal{R}$, implies that $a = 0$. Note that according to \cite[Example 3]{F}, there exists a non-semiprime ring which satisfies \emph{Condition (P)}.

In the following, we will explain the achievements of this article.
Let $\mathcal{A}$ and $\mathcal{B}$ be 2-torsion free rings such that each of them is either semiprime or satisfies \emph{Condition (P)} and let $\mathcal{M}$ be a 2-torsion free faithful $(\mathcal{A}, \mathcal{B})$-bimodule.
Then every Jordan two-sided centralizer on the triangular ring $\mathfrak{T} = Tri(\mathcal{A}, \mathcal{M}, \mathcal{B})$ is a centralizer.
Using several auxiliary results as well as the previous conclusion, we show that every Jordan centralizer on the triangular ring $\mathfrak{T}$ is a centralizer under certain conditions. In view of this theorem, we obtain some previously published results in this regard without assuming unity. For example, we obtain \cite[Corollary 3.2]{G1}. In addition, we prove that every Jordan two-sided generalized derivation and every Jordan generalized derivation on $\mathfrak{T}$ is a two-sided generalized derivation. For more
material about two-sided generalized derivations and other related results, see, e.g. \cite{Hos}.
In this paper, we do not need to assume the rings $\mathcal{A}$ and $\mathcal{B}$ and the module $\mathcal{M}$ be unital, while the above-mentioned papers studying Jordan centralizers on triangular rings (or algebras) assume that $\mathcal{A}$, $\mathcal{B}$ and $\mathcal{M}$ are unital and their proofs depend on the existence of the elements $P = \left [\begin{array}{cc}
\textbf{1}_\mathcal{A} & 0\\
0 & 0
\end{array}\right ]$ and $Q = \left [\begin{array}{cc}
0 & 0\\
0 & \textbf{1}_\mathcal{B}
\end{array}\right ]$, where $\textbf{1}_\mathcal{A}$ and $\textbf{1}_\mathcal{B}$ denote the identity elements of $\mathcal{A}$ and $\mathcal{B}$, respectively. But we prove our results without considering these conditions.

\section{Results and proofs}
We begin with the following useful lemma which we will use frequently to prove Theorems \ref{1}. Recall that a ring $\mathcal{R}$ satisfies \emph{Condition (P)} whenever $x a x = 0$, for all $x \in \mathcal{R}$, implies that $a = 0$.

\begin{lemma}\label{*}\cite[Proposition 1.4]{Z} Let $\mathcal{R}$ be a 2-torsion free ring and let $T:\mathcal{R} \rightarrow \mathcal{R}$ be a Jordan left (resp. right) centralizer. Then for any $x, y, z \in \mathcal{R}$, the following hold:\\
(i) $T(xy + yx) = T(x)y + T(y)x$ $(resp. \ T(xy + yx) = xT(y) + yT(x))$;\\
(ii) $T(xyx) = T(x)yx$ $(resp. \ T(xyx) = xy T(x))$.\\
\end{lemma}
The following result has been motivated by the Main Theorem of \cite{F}. Throughout this article, $\mathcal{M}$ denotes an $(\mathcal{A}, \mathcal{B})$-bimodule such that $\{m \in \mathcal{M} \ | \ \mathcal{A} m =\{0\}\} = \{0\}$ and $\{m \in \mathcal{M} \ | \ m \mathcal{B} =\{0\}\} = \{0\}$. 

\begin{theorem} \label{1}Let $\mathcal{A}$ and $\mathcal{B}$ be 2-torsion free rings such that each of them is either semiprime or satisfies Condition (P) and let $\mathcal{M}$ be a 2-torsion free faithful $(\mathcal{A}, \mathcal{B})$-bimodule.
If $\Psi$ is a Jordan two-sided centralizer on the triangular ring $\mathfrak{T} = Tri(\mathcal{A}, \mathcal{M}, \mathcal{B})$, 
then $\Psi$ is a centralizer on $\mathfrak{T}$.
\end{theorem}
\begin{proof} We first show that $\Psi$ is a left centralizer on $\mathfrak{T} = Tri(\mathcal{A}, \mathcal{M}, \mathcal{B})$.  We prove our claim with an 8-step proof. \\ 

\textbf{Step 1}. It is clear that $\Psi(0) = 0$. \\

\textbf{Step 2}. For every $A_{11} \in \mathfrak{T}_{11}$ and $B_{22} \in \mathfrak{T}_{22}$, we have:\\
(i) $[\Psi(A_{11})]_{22} = 0$ and $[\Psi(B_{22})]_{11} = 0$;\\
(ii) $\Psi(B_{22} A_{11}) = \Psi(B_{22})A_{11}$;\\
(iii) $\Psi( A_{11} B_{22}) = \Psi(A_{11})B_{22}$;\\
(iv) $[\Psi(A_{11})]_{12} = 0$ for all $A_{11} \in \mathfrak{T}_{11}$. \\

Proof. (i) Clearly, $A_{11} B_{22} + B_{22} A_{11} = 0$. Using Lemma \ref{*}, we have
\begin{align*}
0 & = \Psi(A_{11} B_{22} + B_{22} A_{11}) = \Psi(A_{11}) B_{22} + \Psi(B_{22}) A_{11}
\end{align*}
Since $\Psi(0) = 0$, $[\Psi(0)]_{11} =  [\Psi(0)]_{12} = [\Psi(0)]_{22} = 0$. It is clear that $[\Psi(B_{22})A_{11}]_{22} = 0$. Therefore, we have
\begin{align*}
0 & = [\Psi(0)]_{22} = [\Psi(A_{11} B_{22} + B_{22} A_{11})]_{22} \\ & = [\Psi(A_{11}) B_{22} + \Psi(B_{22}) A_{11}]_{22} \\ & = [\Psi(A_{11})]_{22}B_{22},
\end{align*}
which means that $[\Psi(A_{11})]_{22}B_{22} = 0$ and since $\mathcal{B}$ is a semiprime ring or satisfies \emph{Condition (P)}, we infer that $[\Psi(A_{11})]_{22}= 0$ for all $A_{11} \in \mathfrak{T}_{11}$. It is easy to see that $ [\Psi(A_{11}) B_{22}]_{11} = 0$. So, we have
\begin{align*}
0 = [\Psi(0)]_{11} & = [\Psi(A_{11} B_{22} + B_{22} A_{11}]_{11} \\ & = [\Psi(A_{11}) B_{22} + \Psi(B_{22})A_{11}]_{11} \\ & = [\Psi(B_{22})]_{11}A_{11},
\end{align*}
which means that $[\Psi(B_{22})]_{11}A_{11} = 0$ and since $\mathcal{A}$ is a semiprime ring or satisfies \emph{Condition (P)}, we get that $[\Psi(B_{22})]_{11} = 0$ for all $B_{22} \in \mathfrak{T}_{22}$.\\

(ii) First, note that $B_{22}A_{11} = 0$ and so, $\Psi(B_{22} A_{11}) = \Psi(0) = 0$ for all $B_{22} \in \mathfrak{T}_{22}$ and $A_{11} \in \mathfrak{T}_{11}$. Also, we can easily see that
$$\Psi(B_{22})A_{11} = [\Psi(B_{22})]_{11}A_{11} = 0.$$
Hence, we have
\begin{align*}
\Psi(B_{22} A_{11}) = \Psi(B_{22})A_{11}
\end{align*}
for all $B_{22} \in \mathfrak{T}_{22}$ and $A_{11} \in \mathfrak{T}_{11}$.\\

(iii) For any $B_{22} \in \mathfrak{T}_{22}$ and $A_{11} \in \mathfrak{T}_{11}$, we have the following expressions:
\begin{align*}
\Psi(A_{11} B_{22}) = 0 & = \Psi(A_{11} B_{22} + B_{22}A_{11}) \\ & = \Psi(A_{11})B_{22} + \Psi(B_{22})A_{11} \\ & =  \Psi(A_{11})B_{22}.
\end{align*}\\

(iv) Using parts (i) and (iii), we have
\begin{align*}
0 = \Psi(A_{11}B_{22}) = \Psi(A_{11})B_{22} = [\Psi(A_{11})]_{12} B_{22},
\end{align*}
for all $A_{11} \in \mathfrak{T}_{11}$ and $B_{22} \in \mathfrak{T}_{22}$, which means that $[\Psi(A_{11})]_{12} b = 0$ for all $b \in \mathcal{B}$. According to the conditions we assumed for $\mathcal{M}$, we obtain that $[\Psi(A_{11})]_{12} = 0$ for all $A_{11} \in \mathfrak{T}_{11}$. \\

\textbf{Step 3}. $[\Psi(A_{12})]_{11} = 0$ for all $A_{12} \in \mathfrak{T}_{12}$.\\

Proof.  It follows from Lemma \ref{*} (ii) that
$$0 = \Psi(A_{12} B_{11} A_{12}) = \Psi(A_{12})B_{11}A_{12}$$ for all $B_{11} \in \mathfrak{T}_{11}$ and $A_{12} \in \mathfrak{T}_{12}$, and consequently, $[\Psi(A_{12})]_{11} a m = 0$ for all $a \in \mathcal{A}$, $m \in \mathcal{M}$. From this equality we deduce that $[\Psi(A_{12})]_{11} a = 0$ for all $a \in \mathcal{A}$ and since $\mathcal{A}$ is a semiprime ring or satisfies \emph{Condition (P)}, we get that $[\Psi(A_{12})]_{11} = 0$, as desired.\\

\textbf{Step 4}. For any $A_{11} \in \mathfrak{T}_{11}$ and $B_{12} \in \mathfrak{T}_{12}$, the following is true:\\
(i) $\Psi(B_{12} A_{11}) = \Psi(B_{12}) A_{11}$;\\
(ii) $\Psi(A_{11} B_{12}) = \Psi(A_{11})B_{12}$.\\

Proof. (i) By \textbf{Step 3}, we have
\begin{align*}
\Psi(B_{12})A_{11} & = [\Psi(A_{12}]_{11}A_{11} \\ & = 0 \\ & = \Psi(0) \\ & = \Psi(B_{12} A_{11}).
\end{align*}

(ii) Using part (i), we have the following expressions:\\
\begin{align*}
\Psi(A_{11} B_{12}) & = \Psi(A_{11} B_{12} + B_{12} A_{11}) \\ & = \Psi(A_{11})B_{12} + \Psi(B_{12})A_{11} \\ & = \Psi(A_{11})B_{12}.
\end{align*}

\textbf{Step 5.} For any $A_{12} \in \mathfrak{T}_{12}$ and $B_{22} \in \mathfrak{T}_{22}$, we have:\\
(i) $\Psi( A_{12}B_{22}) = \Psi(A_{12}) B_{22}$;\\
(ii) $\Psi(B_{22} A_{12}) = \Psi(B_{22})A_{12}$.\\

Proof. (i) Using \textbf{Step 2}, we have
\begin{align*}
\Psi( A_{12}B_{22}) & = \Psi( A_{12}B_{22} + B_{22} A_{12}) \\ & = \Psi(A_{12}) B_{22} + \Psi(B_{22})A_{12} \\ & = \Psi(A_{12}) B_{22} + [\Psi(B_{22})]_{11}A_{12} \\ & = \Psi(A_{12}) B_{22}.
\end{align*}

(ii) \begin{align*}
\Psi(B_{22}) A_{12} & = [\Psi(B_{22})]_{11}A_{12} \\ & = 0 \\ & = \Psi(B_{22} A_{12}).
\end{align*}

\textbf{Step 6}. $\Psi$ is a left centralizer on $\mathfrak{T}_{12}$.\\

Proof. Let $A_{12}$ and $B_{12}$ be arbitrary elements of $\mathfrak{T}_{12}$. By \textbf{Step 3}, we have
\begin{align*}
\Psi(A_{12})B_{12} & = [\Psi(A_{12})]_{11}B_{12} \\ & = 0 \\ & = \Psi(A_{12} B_{12}).
\end{align*}

\textbf{Step 7}. $\Psi$ is a left centralizer on $\mathfrak{T}_{11}$ and $\mathfrak{T}_{22}$. \\

Proof. First, according to the proofs of \cite[Theorem 2.10]{J1} and \cite[Theorem 2.1]{L1}, for any $A = \left[\begin{array}{cc}
a& m\\
0 & b
\end{array}\right ] \in \mathfrak{T}$, we can write
\begin{align*}
\Psi\left(\left [\begin{array}{cc}
a & m\\
0 & b
\end{array}\right ]\right) = \left[\begin{array}{cc}
f_{11}(a) + g_{11}(m) + h_{11}(b)& f_{12}(a) + g_{12}(m) + h_{12}(b)\\
0 & f_{22}(a) + g_{22}(m) + h_{22}(b)
\end{array}\right ],
\end{align*}
where $f_{11} : \mathcal{A} \rightarrow \mathcal{A}$, $f_{12} : \mathcal{A} \rightarrow \mathcal{M}$, $f_{22} : \mathcal{A} \rightarrow \mathcal{B}$, $g_{11} : \mathcal{M} \rightarrow \mathcal{A}$, $g_{12} : \mathcal{M} \rightarrow \mathcal{M}$,
$g_{22} :\mathcal{M} \rightarrow \mathcal{B}$, $h_{11} : \mathcal{B} \rightarrow \mathcal{A}$, $h_{12} : \mathcal{B} \rightarrow \mathcal{M}$, and $h_{22} : \mathcal{B} \rightarrow \mathcal{B}$ are maps.
Using \textbf{Step 4}, we have
\begin{align*}
\Psi(A_{11} B_{11} C_{12}) = \Psi(A_{11}) B_{11} C_{12}.
\end{align*}
On the other hand, we have
\begin{align*}
\Psi(A_{11} B_{11} C_{12}) = \Psi(A_{11} B_{11})C_{12}.
\end{align*}
Comparing the above expressions, we get that
\begin{align}
\Psi(A_{11} B_{11})C_{12} = \Psi(A_{11}) B_{11} C_{12}.
\end{align}
Using equation (2.1) and the parts (i) and (iv) of \textbf{Step} 2, we have
\begin{align*}
0 & = (\Psi(A_{11} B_{11}) - \Psi(A_{11}) B_{11}) C_{12} \\ & = \left[\begin{array}{cc}
f_{11}(a_1 a_2) + g_{11}(0) + h_{11}(0) & 0\\
0 & 0
\end{array}\right ]\left[\begin{array}{cc}
0& m\\
0 & 0
\end{array}\right ] - \\ & \left[\begin{array}{cc}
f_{11}(a_1) + g_{11}(0) + h_{11}(0)& 0\\
0 & 0
\end{array}\right ] \left[\begin{array}{cc}
a_2& 0\\
0 & 0
\end{array}\right ] \left[\begin{array}{cc}
0& m\\
0 & 0
\end{array}\right ]\\ & = \left[\begin{array}{cc}
0 & (f_{11}(a_1 a_2) + g_{11}(0) + h_{11}(0) - (f_{11}(a_1) + g_{11}(0) + h_{11}(0))a_2)m\\
0 & 0
\end{array}\right ].
\end{align*}
Now, using the assumption that $\mathcal{M}$ is a faithful left $\mathcal{A}$-module, we obtain that
\begin{align*}
f_{11}(a_1 a_2) + g_{11}(0) + h_{11}(0) = (f_{11}(a_1) + g_{11}(0) + h_{11}(0))a_2
\end{align*}
for all $a_1, a_2 \in \mathcal{A}$. Hence, we obtain that
\begin{align*}
\Psi(A_{11} B_{11}) = \Psi(A_{11})B_{11}.
\end{align*}
for all $A_{11}, B_{11} \in X_{11}$. To complete the proof of this step, we must show that $\Psi(A_{22} B_{22}) = \Psi(A_{22})B_{22}$ for all $A_{22}, B_{22} \in X_{22}$. Since $\Psi$ is a Jordan two-sided centralizer on $\mathfrak{T}$, it is observed that $\Psi(X^2) = X \Psi(X)$ for all $ X \in \mathfrak{T}$. Similar to the previous steps, one can easily obtain that $\Psi(C_{12}B_{22}) = C_{12}\Psi(B_{22})$ and $\Psi(A_{11}B_{22}) = A_{11} \Psi(B_{22})$ for all $A_{11} \in \mathfrak{T}_{11}$, $C_{12} \in \mathfrak{T}_{12}$ and $B_{22} \in \mathfrak{T}_{22}$. Using these facts and \textbf{Step 2}(i), we have
\begin{align*}
0 & = \Psi(A_{11}B_{22}) = A_{11}\Psi(B_{22}) \\ & =  \left[\begin{array}{cc}
a & 0\\
0 & 0
\end{array}\right ] \left[\begin{array}{cc}
0& f_{12}(0) + g_{12}(0) + h_{12}(b)\\
0 & f_{22}(0) + g_{22}(0) + h_{22}(b)
\end{array}\right ] \\ & =  \left[\begin{array}{cc}
0 & a(f_{12}(0) + g_{12}(0) + h_{12}(b))\\
0 & 0
\end{array}\right ],
\end{align*}
which means that $a [\Psi(B_{22})]_{12} = 0$ for all $a \in \mathcal{A}$ and according to the assumptions we assume for $\mathcal{M}$, it is obtained that $[\Psi(B_{22})]_{12} = 0$ for all $B_{22} \in \mathfrak{T}_{22}$.
For any $A_{22}, B_{22} \in \mathfrak{T}_{22}$ and $C_{12} \in \mathfrak{T}_{12}$, we have
$\Psi(C_{12}A_{22}B_{22}) = C_{12}\Psi(A_{22}B_{22})$ as well as using \textbf{Step 5}(i), we have $\Psi(C_{12}A_{22}B_{22}) = \Psi(C_{12}A_{22})B_{22} = C_{12}\Psi(A_{22})B_{22}$.
Comparing the previous equations, we deduce that $C_{12}(\Psi(A_{22}B_{22}) - \Psi(A_{22})B_{22}) = 0.$ Using \textbf{Step 2}(i) and the fact that $[\Psi(B_{22})]_{12} = 0$ for all $B_{22}\in X_{22}$, we have
\begin{align*}
0 & = C_{12}(\Psi(A_{22}B_{22}) - \Psi(A_{22})B_{22}) \\ & = \left[\begin{array}{cc}
0& m\\
0 & 0
\end{array}\right ]\left[\begin{array}{cc}
0& 0\\
0 & f_{22}(0) + g_{22}(0) + h_{22}(b_1 b_2)
\end{array}\right ] - \\ & \left[\begin{array}{cc}
0& m\\
0 & 0
\end{array}\right ]\left[\begin{array}{cc}
0& 0\\
0 & f_{22}(0) + g_{22}(0) + h_{22}(b_1)
\end{array}\right ] \left[\begin{array}{cc}
0& 0\\
0 & b_2
\end{array}\right ] \\ & = \left[\begin{array}{cc}
0& m\\
0 & 0
\end{array}\right ]\left[\begin{array}{cc}
0& 0\\
0 & f_{22}(0) + g_{22}(0) + h_{22}(b_1 b_2) - (f_{22}(0) + g_{22}(0) + h_{22}(b_1))b_2
\end{array}\right ] \\ & = \left[\begin{array}{cc}
0& m(f_{22}(0) + g_{22}(0) + h_{22}(b_1 b_2) - (f_{22}(0) + g_{22}(0) + h_{22}(b_1))b_2)\\
0 & 0
\end{array}\right ]
\end{align*}
Using the fact that $\mathcal{M}$ is also a faithful right $\mathcal{B}$-module, one can deduce
\begin{align*}
f_{22}(0) + g_{22}(0) + h_{22}(b_1 b_2) = (f_{22}(0) + g_{22}(0) + h_{22}(b_1))b_2
\end{align*}
for all $b_1 , b_2 \in \mathcal{B}$. Previous equation with the fact that $[\Psi(B_{22})]_{12} = 0$ imply that
\begin{align*}
\Psi(A_{22} B_{22}) = \Psi(A_{22})B_{22}
\end{align*}
for all $A_{22}, B_{22} \in X_{22}$. \\

\textbf{Step 8}. $\Psi$ is a left centralizer on $\mathfrak{T}$. Applying the above steps, we have the following expressions for any $A = A_{11} + A_{12} + A_{22}$ and $B = B_{11} + B_{12} + B_{22}$ in $\mathfrak{T}$,
\begin{align*}
\Psi(AB) & = \Psi\left(\sum_{1 = i \leq j = 2, \ \  1 = k \leq l = 2}A_{ij}B_{kl}\right) \\  & = \sum_{1 = i \leq j = 2, \ \  1 = k \leq l = 2}\Psi\left(A_{ij}B_{kl}\right) \\ & = \sum_{1 = i \leq j = 2, \ \  1 = k \leq l = 2}\Psi(A_{ij})B_{kl} \\ & = \Psi\left(\sum_{1 = i \leq j = 2}A_{ij}\right)B \\ & = \Psi(A)B.
\end{align*}
Now, we are going to prove that $\Psi$ is a right centralizer on $\mathfrak{T}$. We know that $\Psi$ is a Jordan left centralizer. Hence, according to \textbf{Step 2}(iii) and \textbf{Step 4 }(ii), respectively, we have $\Psi(A_{11} B_{22}) = \Psi(A_{11}) B_{22}$ and $\Psi(A_{11}B_{12}) = \Psi(A_{11}) B_{12}$ for all $A_{11} \in \mathfrak{T}_{11}$, $B_{22} \in \mathfrak{T}_{22}$ and $B_{12} \in \mathfrak{T}_{12}$. Reasoning like above, we can show that $\Psi$ is a right centralizer and we leave it to the interested readers. Therefore, $\Psi$ is a centralizer on $\mathfrak{T}$.
\end{proof}

In the following example, we provide the rings $\mathcal{A}$ and $\mathcal{B}$ and also a module $\mathcal{M}$ that satisfy the conditions of Theorem \ref{1}.
\begin{example} Let $\mathbb{Z}$ be the set of all integers. Set
\begin{align*}
\mathcal{A} = \mathcal{B} = \left\{\left [\begin{array}{cc}
2n & 0\\
0 & 2n
\end{array}\right ] \ : \ n \in \mathbb{Z} \right\}.
\end{align*}
It is evident that $\mathcal{A}$ does not contain unity. Let
\begin{align*}
\mathcal{M} = \left\{\left [\begin{array}{cc}
i & j\\
0 & k
\end{array}\right ] \ : \ i, j, k \in \mathbb{Z} \right\}.
\end{align*}
A straightforward verification shows that $\mathcal{A}$ is a semiprime ring satisfying \emph{Condition (P)} and further $\mathcal{M}$ is a faithful $(\mathcal{A}, \mathcal{B})$-bimodule such that $\{m \in \mathcal{M} \ | \ \mathcal{A} m =\{0\}\} = \{0\} = \{m \in \mathcal{M} \ | \ m \mathcal{B} =\{0\}\}$. Therefore, the module $\mathcal{M}$ satisfies all the conditions of Theorem \ref{1}.
\end{example}

The previous theorem proves \cite[Corollary 3.2]{G1} for non-unital triangular rings.\\
\\
By getting idea from \cite{Hos}, an additive mapping $\Phi$ from a ring $\mathcal{R}$ into itself is called a Jordan two-sided generalized derivation associated with a Jordan derivation $\Delta$ on $\mathcal{R}$ if $\Phi(x^2) = \Phi(x) x + x \Delta(x) = x \Phi(x) + \Delta(x)x$ for all $x \in \mathcal{R}$.
\begin{corollary} \label{****} Let $\mathcal{A}$, $\mathcal{B}$, $\mathcal{M}$ and $\mathfrak{T}$ be as Theorem \ref{1} and let $\Phi:\mathfrak{T} \rightarrow \mathfrak{T}$ be a Jordan two-sided generalized derivation associated with a Jordan derivation $\Delta:\mathfrak{T} \rightarrow \mathfrak{T}$. Then $\Phi$ is a two-sided generalized derivation associated with the derivation $\Delta$.
\end{corollary}

\begin{proof} It is clear that $\Psi = \Phi - \Delta$ is a Jordan left- and a right centralizer on $\mathfrak{T}$ and it follows from Theorem \ref{1} that it is a centralizer. Also, it follows from the Main Theorem of \cite{F} that $\Delta$ is a derivation. So, we have
\begin{align*}
\Phi(AB) & = \Delta(AB) + \Psi(AB) \\ & = \Delta(A)B + A \Delta(B) + \Psi(A)B \\ & = \Phi(A)B + A \Delta(B),
\end{align*}
and also
\begin{align*}
\Phi(AB) & = \Delta(AB) + \Psi(AB) \\ & = \Delta(A)B + A \Delta(B) + A \Psi(B) \\ & = A \Phi(B) + \Delta(A)B,
\end{align*}
for all $A, B \in \mathfrak{T}$. It means that $\Phi$ is a two-sided generalized derivation on $\mathfrak{T}$.
\end{proof}

\begin{corollary} Let $\mathcal{T}_{n}$ be the $n \times n$ upper triangular matrix algebra over $\mathbb{C}$. Then every Jordan two-sided generalized derivation associated with a Jordan derivation $\Delta$ on $\mathcal{T}_n$ is a two-sided generalized derivation associated with the derivation $\Delta$.
\end{corollary}
\begin{proof}
Straightforward.
\end{proof}

In the following, we are going to prove that every Jordan centralizer on a triangular ring is a centralizer under certain conditions. Let $\mathcal{R}$ be a ring. An additive mapping $T:\mathcal{R} \rightarrow \mathcal{R}$ is called a Jordan centralizer if $$T(x \circ y) = T(x) \circ y = x \circ T(y)$$ for all $x, y \in \mathcal{R}$, where $x \circ y = xy + yx$. To establish our theorem concerning Jordan centralizers, the following auxiliary results are needed.

\begin{lemma}\label{*1} Let $\mathcal{A}$ and $\mathcal{B}$ be semiprime rings and let $\mathcal{M}$ be a faithful $(\mathcal{A}, \mathcal{B})$-bimodule such that $\{m \in \mathcal{M} \ | \ \mathcal{A} m =\{0\}\} = \{0\}$ or $\{m \in \mathcal{M} \ | \ m \mathcal{B} =\{0\}\} = \{0\}$. If $A_0 X - X A_0 \in Z(\mathfrak{T})$ for all $X \in \mathfrak{T}$, then $A_0 \in Z(\mathfrak{T})$.
\end{lemma}
\begin{proof}Suppose that $\{m \in \mathcal{M} \ | \ \mathcal{A} m =\{0\}\} = \{0\}$. It is easy to that $$Z(\mathfrak{T}) = \Bigg\{\left [\begin{array}{cc}
a & 0\\
0 & b
\end{array}\right ] \ : \ a \in Z(\mathcal{A}), \ b \in Z(\mathcal{B}), am = mb, \  \forall m \in \mathcal{M} \Bigg\}$$
Let $A_0 = \left [\begin{array}{cc}
a_0 & m_0\\
0 & b_0
\end{array}\right ]$ be an element of $\mathfrak{T}$ such that for all $X = \left [\begin{array}{cc}
a & m\\
0 & b
\end{array}\right ] \in \mathfrak{T}$, we have
$$ A_0 X - X A_0 = \left [\begin{array}{cc}
a_0 & m_0\\
0 & b_0
\end{array}\right ]\left [\begin{array}{cc}
a & m\\
0 & b
\end{array}\right ]  - \left [\begin{array}{cc}
a & m\\
0 & b
\end{array}\right ] \left [\begin{array}{cc}
a_0 & m_0\\
0 & b_0
\end{array}\right ] \in Z(\mathfrak{T})$$
Indeed, we have
\begin{align*}
\left [\begin{array}{cc}
a_0 a - a a_0 & a_0 m + m_0 b - a m_0 - m b_0\\
0 & b_0 b - b b_0
\end{array}\right ] \in Z(\mathfrak{T}).
\end{align*}
Hence, we have $a_0 a - a a_0 \in Z(\mathcal{A})$, $b_0 b - b b_0 \in Z(\mathcal{B})$, $a_0 m + m_0 b - a m_0 - m b_0 = 0$ and $(a_0 a - a a_0) m = m (b_0 b - b b_0)$ for all $a \in \mathcal{A}$, $b \in \mathcal{B}$ and $m \in \mathcal{M}$. It follows from \cite[Lemma 2.1]{Z} that $a_0 \in Z(\mathcal{A})$ and $b_0 \in Z(\mathcal{B})$. 
To prove that $A_0 \in Z(\mathfrak{T})$, it is enough to show that $m_0 = 0$ and $a_0 m = m b_0$ for all $m \in \mathcal{M}$. We know that
\begin{align}
a_0 m + m_0 b - a m_0 - m b_0 = 0
\end{align}
for all $a \in \mathcal{A}$, $b \in \mathcal{B}$ and $m \in \mathcal{M}$. Considering $a = 0$ and $b = 0$ in equation (2.2), then we see that $a_0 m = m b_0$ for all $m \in \mathcal{M}$. Letting $m = 0$ in equation (2.2), we have
\begin{align}
m_0 b - a m_0 = 0
\end{align}
for all $a \in \mathcal{A}$ and $b \in \mathcal{B}$. Putting $b = 0$ in equation (2.3), we arrive at $\mathcal{A} m_0 = \{0\}$ and since we are assuming that $\{m \in \mathcal{M} \ | \ \mathcal{A} m =\{0\}\} = \{0\}$, we get that $m_0 = 0$.
Therefore, $a_0 \in Z(\mathcal{A})$, $b_0 \in Z(\mathcal{B})$, $m_0 = 0$ and $a_0 m = m b_0$ for all $m \in \mathcal{M}$. This facts demonstrate that $A_0 = \left [\begin{array}{cc}
a_0 & m_0\\
0 & b_0
\end{array}\right ]  = \left [\begin{array}{cc}
a_0 & 0\\
0 & b_0
\end{array}\right ] \in Z(\mathfrak{T})$, as desired. Using the same argument as stated, we can obtain our result if we suppose that $\{m \in \mathcal{M} \ | \ m \mathcal{B} =\{0\}\} = \{0\}$.
\end{proof}

\begin{lemma}\label{*2} Let $\mathcal{A}$, $\mathcal{B}$ and $\mathcal{M}$ be as Lemma \ref{*1} and let $A_0$ be a fixed element of $\mathfrak{T}$. If $\Psi:\mathfrak{T} \rightarrow \mathfrak{T}$ defined by $\Psi(X) = A_0 X + X A_0$ is a Jordan centralizer, then $A_0 \in Z(\mathfrak{T})$.
\end{lemma}
\begin{proof}
We know that $\Psi(XY + YX) = \Psi(X)Y + Y \Psi(X)$ for all $X, Y \in \mathfrak{T}$. Therefore, we have
\begin{align*}
A_0 XY + A_0 YX + X Y A_0 + YX A_0 = A_0 XY + X A_0 Y + Y A_0 X + YX A_0
\end{align*}
for all $X, Y \in \mathfrak{T}$. Previous equation gives us:
\begin{align*}
0 & =  A_0 YX + X Y A_0 - X A_0 Y - Y A_0 X \\ & = (A_0 Y - YA_0)X - X(A_0 Y - YA_0)
\end{align*}
for all $X, Y \in \mathfrak{T}$. This means that $A_0 Y - YA_0 \in Z(\mathfrak{T})$ for all $Y \in \mathfrak{T}$. Then in view of Lemma \ref{*1}, $A_0 \in Z(\mathfrak{T})$. This proves the lemma.
\end{proof}

\begin{lemma}\label{*3} Let $\mathcal{A}$, $\mathcal{B}$ and $\mathcal{M}$ be as Lemma \ref{*1}. Then every Jordan centralizer on $\mathfrak{T}$ maps $Z(\mathfrak{T})$ into itself.
\end{lemma}

\begin{proof}Let $C$ be an arbitrary fixed element of $Z(\mathfrak{T})$ and let $\Psi: \mathfrak{T} \rightarrow \mathfrak{T}$ be a Jordan centralizer. We have to show that $\Psi(C) \in Z(\mathfrak{T})$. Letting $A_0 = \Psi(C)$, we have
\begin{align*}
2\Psi(CX) & = \Psi(2CX) \\ & = \Psi(CX + XC) \\ & = \Psi(C) X + X\Psi(C) \\ & = A_0 X + X A_0
\end{align*}
We define a mapping $\Lambda: \mathfrak{T} \rightarrow \mathfrak{T}$ by $\Lambda(X) = 2\Psi(CX)$. We have the following expressions:
\begin{align*}
\Lambda(XY + YX) & = 2\Psi(C(XY + YX)) \\ & = 2 \Psi(CXY + YCX) \\ & = 2 \Psi(CX)Y + 2 Y \Psi(CX) \\ & = \Lambda(X)Y + Y \Lambda(X)
\end{align*}
for all $X, Y \in \mathfrak{T}$, which means that $\Lambda(X) = A_0 X + X A_0$ is a Jordan centralizer. By Lemma \ref{*2}, $A_0 = \Psi(C) \in Z(\mathfrak{T})$, as desired.
\end{proof}

\begin{lemma}\label{*4} Let $\mathcal{A}$, $\mathcal{B}$ and $\mathcal{M}$ be as Lemma \ref{*1}. If $A_0 \in Z(\mathfrak{T})$ and $A_0^{2} = 0$, then $A_0 = 0$.
\end{lemma}

\begin{proof}Let $A_0 = \left [\begin{array}{cc}
a_0 & 0\\
0 & b_0
\end{array}\right ] \in Z(\mathfrak{T})$ such that $A_0^{2} = 0$. We know that $A_0 \in Z(\mathfrak{T})$, $a_0 \in Z(\mathcal{A})$, $b_0 \in Z(\mathcal{B})$, $a_0 m = m b_0$ for all $m \in \mathcal{M}$ and also $A_0^{2} = 0$. So, we have $a_0^{2} = 0$ and $b_0^{2} = 0$. Note that $0 = a a_0^{2} = a_0 a a_0$ for all $a \in \mathcal{A}$ and consequently, $a_0 = 0$, since $a_0 \in Z(\mathcal{A})$ and $\mathcal{A}$ is semiprime. Similarly, we can show that $b_0 = 0$ and it implies that $A_0 = 0$.
\end{proof}

We are now in a position to establish another main theorem of this article which has been motivated by \cite{Z}.

\begin{theorem}\label{*5} Let $\mathcal{A}$ and $\mathcal{B}$ be 2-torsion free semiprime rings and let $\mathcal{M}$ be as Theorem \ref{1}. Then every Jordan centralizer of $\mathfrak{T}$ is a centralizer.
\end{theorem}

\begin{proof} Let $\Psi:\mathfrak{T} \rightarrow \mathfrak{T}$ be a Jordan centralizer, i.e.
\begin{align}
\Psi(XY + YX) = \Psi(X)Y + Y\Psi(X) = X\Psi(Y) + \Psi(Y)X
\end{align}
for all $X, Y \in \mathfrak{T}$. If we replace $Y$ by $XY + YX$ in (2.4), we get
\begin{align*}
\Psi(X)(XY + YX) + (XY + YX)\Psi(X) & = \Psi(XY + YX)X + X\Psi(XY +YX) \\ & = (\Psi(X)Y + Y\Psi(X))X + X(\Psi(X)Y \\ & + Y\Psi(X)).
\end{align*}
Hence, we deduce that $[\Psi(X), X]Y = Y[\Psi(X), X]$ for all $X, Y \in \mathfrak{T}$, which means that $[\Psi(X), X] \in Z(\mathfrak{T})$ for any $X \in \mathfrak{T}$. Our next task is to show that $[\Psi(X), X] = 0$ for all $X \in \mathfrak{T}$. For any $C \in Z(\mathfrak{T})$ and $X \in \mathfrak{T}$, we have
$$2 \Psi(XC) = \Psi(XC + CX) = \Psi(X) C + C \Psi(X) =  2\Psi(X)C,$$
which means that
\begin{align}
\Psi(XC) = \Psi(X)C, \ \ \ \ \ (X \in \mathfrak{T})
\end{align}

Using Lemma \ref{*3}, we get that
$$2 \Psi(CX) = \Psi(XC + CX) = \Psi(C)X + X\Psi(C) =  2\Psi(C)X,$$
which means that
\begin{align}
\Psi(XC) = \Psi(CX) = \Psi(C)X, \ \ \ \ \ (X \in \mathfrak{T})
\end{align}
Comparing (2.5) and (2.6), we infer that
\begin{align}
\Psi(X)C = \Psi(CX) = \Psi(C)X, \ \ \ \ \ (X \in \mathfrak{T})
\end{align}
Using (2.7) and Lemma \ref{*3}, we have
\begin{align*}
[\Psi(X), X]C & = \Psi(X) XC - X\Psi(X)C \\ & = \Psi(X)CX - X\Psi(X)C \\ & = \Psi(C)X^{2} - X\Psi(C)X \\ & = \Psi(C)X^{2} - \Psi(C)X^{2} \\ & = 0
\end{align*}
for any $X \in \mathfrak{T}$ and $C \in Z(\mathfrak{T})$. From previous equation and using the fact that $[\Psi(X), X] \in Z(\mathfrak{T})$ for all $X \in \mathfrak{T}$, we get that $[\Psi(X), X]^{2} = 0$ for all $X \in \mathfrak{T}$. It follows from Lemma \ref{*4} that $[\Psi(X), X] = 0$ for all $X \in \mathfrak{T}$. So, we have
\begin{align*}
2 \Psi(X^2) & = \Psi(XX + XX) \\ & = \Psi(X)X + X\Psi(X) \\ & = 2\Psi(X)X,
\end{align*}
which implies that $\Psi(X^2) = \Psi(X)X$ for all $X \in \mathfrak{T}$. Simimlarly, we can obtain that $\Psi(X^2) = X\Psi(X)$ for all $X \in \mathfrak{T}$. It means that $\Psi$ is a Jordan two-sided centralizer on $\mathfrak{T}$. In view of Theorem \ref{1}, we get that $\Psi$ is a centralizer on $\mathfrak{T}$.
\end{proof}

Immediate consequence of Theorem \ref{*5} reads as follows.

\begin{corollary}\label{*6} Let $\mathcal{A}$, $\mathcal{B}$ and $\mathcal{M}$ be as Theorem \ref{*5} and let $\Phi:\mathfrak{T} \rightarrow \mathfrak{T}$ be a Jordan generalized derivation associated with a Jordan derivation $\Delta$, that is $\Phi(X \circ Y) = \Phi(X) \circ Y + X \circ \Delta(Y)$ and $\Delta(X \circ Y) = \Delta(X) \circ Y + X \circ \Delta(Y)$ for all $X, Y \in \mathfrak{T}$. Then $\Phi$ is a two-sided generalized derivation associated with the derivation $\Delta$.
\end{corollary}

\begin{proof} Obviously, $\Psi = \Phi - \Delta$ is a Jordan centralizer on $\mathfrak{T}$. Using Theorem \ref{*5}, we infer that $\Psi$ is a centralizer and it follows from the Main Theorem of \cite{F} that $\Delta$ is a derivation on $\mathfrak{T}$. Therefore, we have
\begin{align*}
\Phi(XY) & = \Delta(XY) + \Psi(XY) \\ & = \Delta(X) Y + X \Delta(Y) + \Psi(X)Y \\&  = \Phi(X)Y + X \Delta(Y),
\end{align*}
and also
\begin{align*}
\Phi(XY) & = \Delta(XY) + \Psi(XY) \\ & = \Delta(X) Y + X \Delta(Y) + X\Psi(Y) \\&  = \Delta(X)Y + X \Phi(Y),
\end{align*}
for all $X, Y \in \mathfrak{T}$, which means that $\Phi$ is a two-sided generalized derivation on $\mathfrak{T}$. This is the required result.
\end{proof}

\bibliographystyle{amsplain}

\end{document}